\documentclass[12pt]{article}

\usepackage[english]{babel}
\usepackage[utf8x]{inputenc}
\usepackage{amsmath}
\usepackage{amsthm}
\usepackage{stmaryrd}
\usepackage{amssymb}
\usepackage{tikz-cd}
\usepackage{graphicx}
\usepackage{authblk}
\usepackage[colorinlistoftodos]{todonotes}
\usepackage{geometry}
\usepackage{indentfirst}
\numberwithin{equation}{section}

\newtheorem{theorem}{Theorem}
\newtheorem{definition}{Definition}
\newtheorem{example}{Example}
\newtheorem{proposition}{Proposition}
\newtheorem{remark}{Remark}
\newtheorem{corollary}{Corollary}
\newtheorem{lemma}{Lemma}

\newcommand{\Addresses}{{% additional braces for segregating \footnotesize
  \bigskip
  \footnotesize

  Junhua Zheng, \textsc{School of Mathematical Sciences Capital Normal University,
    }\par\nopagebreak
  \textit{E-mail address}, Junhua Zheng: \texttt{2200502047@cnu.edu.cn}

  \medskip

}}

\title{A correspondence and distance of t-structures}
\author{Junhua Zheng}
\date{}

\begin{document}

\maketitle

\begin{abstract}
For two t-structures $D_{1}=(D_{1}^{\leqslant 0},D_{1}^{\geqslant 1})$ and $D_{2}=(D_{2}^{\leqslant 0},D_{2}^{\geqslant 1})$ with $D_{1}^{\leqslant 0} \subseteq D_{2}^{\leqslant 0}$ on a triangulated category $\mathcal{D}$, we give a correspondence between t-structure $D_{i}=(D_{i}^{\leqslant 0},D_{i}^{\geqslant 1})$ which satisfies $D_{1}^{\leqslant 0} \subseteq D_{i}^{\leqslant 0} \subseteq D_{2}^{\leqslant 0}$ and a pair of full subcategories of $D_{1}^{\geqslant 1}\bigcap D_{2}^{\leqslant 0}$. Then we give a way to determine the distance of two t-structure if we have known that their distance is finite.In addition, if we set a t-structure $D_{1}$ whose heart $H_{1} \neq 0$ and that $H_{1}$ has a non-trivial torsion pair, then for any integer $n$, we can construct a t-structure $D_{2}$ such that the distance between $D_{1}$ and $D_{2}$ is $n$.
\end{abstract}

\section{Introduction}
The notion of t-structure was introduced in \cite{deligne1983faisceaux}. t-structures are the tool which allows us to see the different abelian categories embedded in a given triangulated category. They have been widely studied in various contexts including representation theory, silting theory, cluster theory and stability conditions(\cite{bridgeland2007stability}, \cite{qiu2018contractible}).

\par Let $\mathcal{D}$ be a triangulated category. We denote $D_{i}=(D_{i}^{\leqslant 0},D_{i}^{\geqslant 1})$ for t-structure on $\mathcal{D}$ and $H_{i}$ for the heart of $D_{i}$. Suppose that $D_{1}=(D_{1}^{\leqslant 0},D_{1}^{\geqslant 1})$ and $D_{2}=(D_{2}^{\leqslant 0},D_{2}^{\geqslant 1})$ are t-structures on $\mathcal{D}$ with $D_{1}^{\leqslant 0} \subseteq D_{2}^{\leqslant 0}$. We can also define t-structure on $D_{1}^{\geqslant 1}\bigcap D_{2}^{\leqslant 0}$, which will be introduced in section 3. The first main result of this paper is to give a correspondence as follow.

\begin{theorem}
There is a correspondence:
$$\begin{array}{*{3}{rll}}
\text{\{ t-structure \,} D_{i}=(D_{i}^{\leqslant 0},D_{i}^{\geqslant 1}) \text{\, on \,} \mathcal{D} |\text{\,} D_{1}^{\leqslant 0} \subseteq D_{i}^{\leqslant 0} \subseteq D_{2}^{\leqslant 0}\} & \leftrightarrow & \{\text{t-structure on \,} \mathcal{H} \} \\
(D_{i}^{\leqslant 0},D_{i}^{\geqslant 1}) & \mapsto & (\mathcal{H} \bigcap D_{i}^{\leqslant 0},D_{i}^{\geqslant 1}\bigcap \mathcal{H})\\
(D_{1}^{\leqslant 0}\ast U_{i}^{\leqslant 0}, U_{i}^{\geqslant 1}\ast D_{2}^{\geqslant 1}) & \mapsfrom & (U_{i}^{\leqslant 0}, U_{i}^{\geqslant 1}).
\end{array}$$
where $\mathcal{H}=D_{1}^{\geqslant 1}\bigcap D_{2}^{\leqslant 0}$
\end{theorem}

This correspondence has a application in the distance between two t-structures when we let $D_{2}=(D_{1}^{\leqslant n},D_{1}^{\geqslant n+1})$ for some natural number $n$. The distance between two t-structures is developed in \cite{chen2022extensions}(Section 4) to describe a result of extending t-structure in (\cite{keller2005triangulated},Subsection 6.1). The distance between two t-structures has a connection to global dimension of t-structure (\cite{chen2022extensions}Section 3). For two t-structures $D_{1}=(D_{1}^{\leqslant 0},D_{1}^{\geqslant 1})$ and $D_{2}=(D_{2}^{\leqslant 0},D_{2}^{\geqslant 1})$, we denote $d(D_{1},D_{2})$ to be the distance between $D_{1}$ and $D_{2}$, which will be introduced in Section\ref{section distance}. 
\par The second main result of this paper is to describe how to determine the distance of two t-structures. Let $D_{1}=(D_{1}^{\leqslant 0},D_{1}^{\geqslant 1})$ and $D_{2}=(D_{2}^{\leqslant 0},D_{2}^{\geqslant 1})$ are t-structures on $\mathcal{D}$ such that $D_{1}^{\leqslant n_{1}} \subseteq D_{2}^{\leqslant 0} \subseteq D_{1}^{\leqslant n_{2}}$ for some integers $n_{1}$, $n_{2}$ with $n_{1} \leqslant n_{2}$. We denote $U=(U^{\leqslant 0},U^{\geqslant 1})$ the corresponding t-structure of $D_{2}$ on  $D_{1}^{\geqslant n_{1}+1}\cap D_{1}^{\leqslant n_{2}}$.

\begin{theorem}
If there exist integer number $i \in \{n_{1}+1,\cdots, n_{2}\}$ such that for any $n_{1}+1 \leqslant j \leqslant i$, $\Sigma^{i}(H_{1})\subseteq U^{\leqslant 0}$(resp.$j \geqslant i$, $\Sigma^{i}(H_{1})\subseteq U^{\geqslant 1}$), then we suppose that  $a$(resp. $b$) is the largest(resp. smallest) one; otherwise, we set $a=n_{1}$ and $b=n_{2}+1$.Then we have $d(D_{1},D_{2})=b-a-1$.
\end{theorem} 

With the help of the above results, we have the following proposition:

\begin{proposition}
Let $D_{1}=(D_{1}^{\leqslant 0}, D_{1}^{\geqslant 1})$ be a t-structure satisfying that the heart $H_{1}$ is non-zero and there is a non-trivial torsion pair $(T,F)$ in $H_{1}$. Then for any natural number $n$, we have a t-structure $D_{2}=(D_{2}^{\leqslant 0}, D_{2}^{\geqslant 1})$ such that $d(D_{1},D_{2})=n$.
\end{proposition}

The paper is organised as follow. In Section 2, we recall some basic definitions and properties of t-structure on a triangulated category and give some lemmas which are frequently used in this paper. In Section 3, we study a correspondence between special t-structures and a pair of full subcategories, which is the first main result in this paper. In Section 4, we discuss the distance between two t-structures and some results on distance of t-structures.

\section{Preliminaries on t-structures}
In this section, we give the definition and some basic properties of t-structure and fix some notations. Most of them can be found in \cite{deligne1983faisceaux}.
\par Through this paper, we always let $\mathcal{D}$ be a triangulated category, and we denote $\Sigma$ its translation and ${\Sigma}^{-1}$ its quasi-inverse. Consequently, ${\Sigma}^{n}$ is defined for each integer $n$.
%\par For a full subcategory $\mathcal{X}$ of $\mathcal{D}$, the right orthogonal subcategory is defined as 
%$${\mathcal{X}}^{\perp}=\{Z \in \mathcal{D} | \text{\, Hom}_{\mathcal{D}}(X,Z)=0 \text{\, for any \,} X \in \mathcal{X} \}$$
%The left orthogonal subcategory $^{\perp}{\mathcal{X}}$ is defined analogously. For another full subcategory $\mathcal{Y}$ of $\mathcal{D}$, 
\par For two full subcategories $\mathcal{X}$ and $\mathcal{Y}$ of $\mathcal{D}$ ,we set 
$$\mathcal{X} \ast \mathcal{Y}=\{Z \in \mathcal{D}| \, \exists \text{an distinguished triangle \,} X\rightarrow Z \rightarrow Y \rightarrow \Sigma(X) \text{\, with \,} X \in \mathcal{X}, Y \in \mathcal{Y} \}$$
The operation $\ast$ is associative (\cite{deligne1983faisceaux}, Lemma 1.3.10 ). By $\text{Hom}_{\mathcal{D}}(\mathcal{X}, \mathcal{Y})=0$, we mean that $\text{Hom}_{\mathcal{D}}(X, Y)=0$ for any $X \in \mathcal{X}$ and $Y \in \mathcal{Y}$. We should mention that $\mathcal{X} \ast \mathcal{Y}$ is closed under isomorphism classes, which is an important fact.

\par Recall that a $\textbf{t-structure}$ $D_{1}=(D_{1}^{\leqslant 0},D_{1}^{\geqslant 1})$ of $\mathcal{D}$ consists of two full subcategories, which are subject to the following conditions:
\begin{itemize}
\item[(T1)] $\Sigma D_{1}^{\leqslant 0}\subseteq D_{1}^{\leqslant 0}$ and $D_{1}^{\geqslant 1}\subseteq  \Sigma D_{1}^{\geqslant 1}$;
\item[(T2)] $\text{Hom}_{\mathcal{D}}(D_{1}^{\leqslant 0},D_{1}^{\geqslant 1})=0$;
\item[(T3)] for any object $X \in \mathcal{D}$, there is a distinguished triangle 
$$X_{0} \rightarrow X \rightarrow X_{1} \rightarrow \Sigma(X_{0})$$
for some $X_{0} \in D_{1}^{\leqslant 0}$ and $X_{1} \in D_{1}^{\geqslant 1}$.
\end{itemize}
For each integer $n$, we set $D_{1}^{\leqslant n}=\Sigma^{-n} D_{1}^{\leqslant 0}$, $D_{1}^{\geqslant n}=\Sigma^{-n+1} D_{1}^{\geqslant 0}$ and $\Sigma^{-n}(D_{1})=(D_{1}^{\leqslant n},D_{1}^{\geqslant 1}$. The $\textbf{heart}$ of the t-structure is defined to be $H_{1}=D_{1}^{\leqslant 0} \bigcap D_{1}^{\geqslant 0}$, which is an abelian category(\cite{deligne1983faisceaux}, Theorem 1.3.6).
\begin{remark}
A t-structure is called \textbf{stable}, if $\Sigma (D_{1}^{\leqslant 0})=D_{1}^{\leqslant 0}$, i.e. $D_{1}^{\leqslant 0}$ is a triangulated subcategory of $D$. When a t-structure $D_{1}=(D_{1}^{\leqslant 0},D_{1}^{\geqslant 1})$ is not stable, we will have $m \leqslant n$ if and only if $D_{1}^{\leqslant m} \subseteq D_{1}^{\leqslant n}$.
\end{remark}

For a t-structure $D_{1}=(D_{1}^{\leqslant 0},D_{1}^{\geqslant 1})$, we denote ${\tau}_{1}^{\leqslant 0}: \mathcal{D} \rightarrow D_{1}^{\leqslant 0}$ the right adjoint of the inclusion $D_{1}^{\leqslant 0} \hookrightarrow \mathcal{D}$ and by ${\tau}_{1}^{\geqslant 1}: \mathcal{D} \rightarrow D_{1}^{\geqslant 1}$ the left adjoint of the inclusion $D_{1}^{\geqslant 1} \hookrightarrow \mathcal{D}$. They are called the $\textbf{truncation functors}$ associated to $D_{1}$, and for convenience, we denote $X_{1}^{\leqslant 0}={\tau}_{1}^{\leqslant 0}(X)$  and $X_{1}^{\geqslant 1}={\tau}_{1}^{\geqslant 0}(X)$ for any $X \in \mathcal{D}$. Hence for any $X \in \mathcal{D}$, there is a canonical distinguished triangle
$$X_{1}^{\leqslant 0} \rightarrow X \rightarrow X_{1}^{\geqslant 1} \rightarrow \Sigma(X_{1}^{\leqslant 0}).$$
Notice the decomposition is unique up to isomorphism, therefore this distinguished triangle is called  \textbf{the $D_{1}$-canonical distinguished triangle} of $X$

%\begin{proposition}\label{proposition}
%Let $D_{1}=(D_{1}^{\leqslant 0},D_{1}^{\geqslant 1})$ be a t-structure in $\mathcal{D}$. If there is a distinguished triangle $Y \rightarrow X \rightarrow Z \rightarrow \Sigma(Y)$ with $Y \in D_{1}^{\leqslant 0}$ and $Z \in D_{1}^{\geqslant 1}$, then we have $Y \cong X_{1}^{\leqslant 0}$ and $Z \cong X_{1}^{\geqslant 1}$
%\end{proposition}
%\begin{proof}
%Applying the Hom-functors $Hom_{\mathcal{D}}(Y,-)$ and $Hom_{\mathcal{D}}(X_{1}^{\leqslant 0},-)$ for the distinguished triangles respectively:
%$$X_{1}^{\leqslant 0} \rightarrow X \rightarrow X_{1}^{\geqslant 1} \rightarrow \Sigma(X_{1}^{\leqslant 0}),$$
%$$Y \rightarrow X \rightarrow Z \rightarrow \Sigma(Y)$$

%we get two exact sequences 
%$$\quad 0 \rightarrow Hom_{\mathcal{D}}(Y,X_{1}^{\leqslant 0}) \rightarrow Hom_{\mathcal{D}}(Y,X) \rightarrow 0$$  
%$$\text{and \quad} 0 \rightarrow Hom_{\mathcal{D}}(X_{1}^{\leqslant 0},Y) \rightarrow Hom_{\mathcal{D}}(X_{1}^{\leqslant 0},X) \rightarrow 0,$$
%which show that we have a canonical commutative diagram and the column morphisms are isomorphisms
%$$\begin{tikzcd}
%Y \arrow[r] \arrow[d]         & X \arrow[r] \arrow[d, "id"] & Z \arrow[r] \arrow[d]       & \Sigma(Y) \arrow[d]         \\
%X_{1}^{\leqslant 0} \arrow[r] & X \arrow[r]                 & X_{1}^{\geqslant} \arrow[r] & \Sigma(X_{1}^{\leqslant 0})
%\end{tikzcd}$$
%\end{proof}

\begin{remark}\label{useful remark}
In general, for a pair of full categories $(\mathcal{X},\mathcal{Y})$, they satisfy 
$$\text{Hom}_{\mathcal{D}}(\mathcal{X},\mathcal{Y})=\text{Hom}_{\mathcal{D}}(\mathcal{X},\Sigma^{-1}(\mathcal{Y}))=\text{Hom}_{\mathcal{D}}(\Sigma(\mathcal{X}),\mathcal{Y})=0,$$
then the object in $\mathcal{X} \ast \mathcal{Y}$ has a unique decomposition up to isomorphism, and the proof is the same like the proof of the t-structure.
\end{remark}

\par The composition $H_{1}^{0}={\tau}_{1}^{\leqslant 0} {\tau}_{1}^{\geqslant 0}: \mathcal{D} \rightarrow H_{1}$ is the $\textbf{cohomological functor}$ associated to $D_{1}$. More generally, we set $H_{1}^{n}(X)=H_{1}^{0}(\Sigma^{n}(X))$, which is canonically isomorphic to $\Sigma^{n} {\tau}_{1}^{\leqslant n}{\tau}_{1}^{\geqslant n}$.

\par Before the end of this section, we list some lemmas that would be used many times in the rest. The following lemma named $3\times 3$ lemma is well-known, hence we omit the proof.

\begin{lemma}[$3\times 3 \,lemma$,\cite{may2001additivity}Lemma 2.6]
Assume that $j \circ f=f' \circ i$ and the two top rows and two  two left columns are distinguished triangles in the following diagram.
$$\begin{tikzcd}
X \arrow[r, "f"] \arrow[d, "i"]               & Y \arrow[r, "g"] \arrow[d, "j"]               & Z \arrow[r, "h"] \arrow[d, "k", dotted]               & \Sigma(X) \arrow[d, "\Sigma(i)"]      \\
X' \arrow[d, "i'"] \arrow[r, "f'"]            & Y' \arrow[r, "g'"] \arrow[d, "j'"]            & Z' \arrow[r, "h'"] \arrow[d, "k'", dotted]            & \Sigma(X') \arrow[d, "\Sigma(i')"]    \\
X'' \arrow[d, "i''"] \arrow[r, "f''", dotted] & Y'' \arrow[d, "j''"] \arrow[r, "g''", dotted] & Z'' \arrow[d, "k''", dotted] \arrow[r, "h''", dotted] & \Sigma(X'') \arrow[d, "-\Sigma(i'')"] \\
\Sigma(X) \arrow[r, "\Sigma(f)"]              & \Sigma(Y) \arrow[r, "\Sigma(g)"]              & \Sigma(Z) \arrow[r, "-\Sigma(h)"]                     & \Sigma^{2}(X)                        
\end{tikzcd}$$
Then there is an object $Z''$ and there are dotted arrow maps $f'',g'',h'',k,k',k''$ such that the diagram is commutative except for its bottom right square, which commutes up to the sign -1, and all four rows and columns are distinguished triangles.
\end{lemma} 

By using $3\times 3$ lemma, we can construct a new distinguished triangle, which will be used several times in the other sections. And the special case is the following statement.
\begin{corollary}\label{cor3x3}
If we have the following commutative diagram with all columns and the first and the second rows are distinguished triangles,
$$\begin{tikzcd}
X \arrow[r, "id"] \arrow[d] & X \arrow[r] \arrow[d]     & 0 \arrow[d]           \\
Y_{1} \arrow[r] \arrow[d]   & Y_{2} \arrow[r] \arrow[d] & Y_{3} \arrow[d, "id"] \\
Z_{1} \arrow[r]             & Z_{2} \arrow[r]           & Y_{3}                
\end{tikzcd}$$
then the third row is distinguished triangle as well.
\end{corollary}
\begin{proof}
By using $3\times 3$ lemma, there is a object $U$ such that the following diagram is commutative except for its bottom right square, which commutes up to the sign -1, and all four rows and columns are distinguished triangles. 
$$\begin{tikzcd}
X \arrow[r, "id"] \arrow[d] & X \arrow[r] \arrow[d]     & 0 \arrow[d] \arrow[r]     & \Sigma(X) \arrow[d]     \\
Y_{1} \arrow[r] \arrow[d]   & Y_{2} \arrow[r] \arrow[d] & Y_{3} \arrow[d] \arrow[r] & \Sigma(Y_{1}) \arrow[d] \\
Z_{1} \arrow[r] \arrow[d]   & Z_{2} \arrow[r] \arrow[d] & U \arrow[r] \arrow[d]     & \Sigma(Z_{1}) \arrow[d] \\
\Sigma(X) \arrow[r]         & \Sigma(X) \arrow[r]       & 0 \arrow[r]               & \Sigma^{2}(X)          
\end{tikzcd}$$
Then $U \cong Y_{3}$ and we can replace $U$ with $Y_{3}$, and we finish the proof.
\end{proof}

%\begin{lemma}
%Let $D_{i}=(D_{i}^{\leqslant 0},D_{i}^{\geqslant 1})$ be a t-structure on $\mathcal{D}$. And we suppose that $X \rightarrow Y \rightarrow Z \rightarrow \Sigma(X)$ is a distinguished triangle in $\mathcal{D}$.
%\begin{itemize}
%\item[(1)] If $Z \in D_{i}^{\geqslant 0}$, then we have a exact sequence $0 \rightarrow H_{i}^{0}(X) \rightarrow H_{i}^{0}(Y) \rightarrow H_{i}^{0}(Z)$;
%\item[(2)] If $X \in D_{i}^{\leqslant 0}$, then we have a exact sequence $H_{i}^{0}(X) \rightarrow H_{i}^{0}(Y) \rightarrow H_{i}^{0}(Z) \rightarrow 0$.
%\end{itemize}
%\end{lemma}

\section{Correspondence between t-structures}
In this section, we develop a correspondence between some particular t-structures and some pairs of full subcategories.
\par Through this section, we let $D_{1}=(D_{1}^{\leqslant 0},D_{1}^{\geqslant 1})$ and $D_{2}=(D_{2}^{\leqslant 0},D_{2}^{\geqslant 1})$ be a pair of t-structures with $D_{1}^{\leqslant 0} \subseteq D_{2}^{\leqslant 0}$,which is equivalent to $D_{1}^{\geqslant 1} \supseteq D_{2}^{\geqslant 1}$.

\begin{definition}
Let $D_{1}$ and $D_{2}$ be as above. A \textbf{t-structure} $(U_{1}^{\leqslant 0}, U_{1}^{\geqslant 1})$ of  $D_{2}^{\geqslant 1} \bigcap D_{1}^{\leqslant 0}$ consists of two full subcategories, which are subject to the following conditions:
\begin{itemize}
\item[(T1')]$D_{1}^{\leqslant 0}\ast U_{1}^{\leqslant 0} \subseteq D_{1}^{\leqslant 1}\ast U_{1}^{\leqslant 1}$ and $U_{1}^{\geqslant 1} \ast D_{2}^{\geqslant 1}  \subseteq U_{1}^{\geqslant 0} \ast D_{2}^{\geqslant 0} $;
\item[(T2')] $\text{Hom}_{\mathcal{D}}(U_{1}^{\leqslant 0},U_{1}^{\geqslant 1})=0$;
\item[(T3')] for any object $X \in D_{2}^{\geqslant 1} \bigcap D_{1}^{\leqslant 0}$, there is a distinguished triangle 
$$Y \rightarrow X \rightarrow Z \rightarrow \Sigma(H_{0})$$
for some $Y \in U_{1}^{\leqslant 0}$ and $Z \in U_{1}^{\geqslant 1}$.
\end{itemize}
For each integer n, we also define $U_{1}^{\leqslant n}=\Sigma^{-n}(U_{1}^{\leqslant 0})$ and $U_{1}^{\leqslant n}=\Sigma^{-n+1}(U_{1}^{\geqslant 1})$.
\end{definition}

This definition is similar to the definition of t-structure on $\mathcal{D}$ . It also make a connection to torsion pair in heart of a t-structure. In order not to create ambiguity, for any t-structure $U_{1}$  and object $X$ of $D_{2}^{\geqslant 1} \bigcap D_{1}^{\leqslant 0}$, the canonical distinguished triangle would be denoted by

$$X_{U_{1}}^{\leqslant 0} \rightarrow X \rightarrow X_{U_{1}}^{\geqslant 1} \rightarrow \Sigma(X_{U_{1}}^{\leqslant 0})$$
with $X_{U_{1}}^{\leqslant 0} \in U_{1}^{\leqslant 0}$ and $X_{U_{1}}^{\geqslant 1} \in U_{1}^{\geqslant 1}$ and called the $U_{1}$-canonical distinguished triangle of $X$.

\begin{example}
Recall that a torsion pair $(\mathcal{T}, \mathcal{F})$ of an abelian category $\mathcal{A}$ consists of two full subcategories satisfying the following properties:
\begin{itemize}
\item[(A1)] $\text{Hom}_{\mathcal{A}}(\mathcal{T},\mathcal{F})=0$;
\item[(A2)] for any $M \in \mathcal{A}$, there exists a short exact sequence
$$0 \rightarrow T_{M} \rightarrow M \rightarrow F_{M} \rightarrow 0$$
with $T_{M} \in \mathcal{T}$ and $F_{M} \in \mathcal{F}$. 

\end{itemize}
A torsion pair is called \textbf{trivial} if $\mathcal{A}=\mathcal{T}$ or $\mathcal{A}=\mathcal{F}$.Let's consider two special t-structures: $(D_{1}^{\leqslant -1},D_{1}^{\geqslant 0})$ and $(D_{1}^{\leqslant 0},D_{1}^{\geqslant 1}).$ Then $D_{1}^{\geqslant 0} \bigcap D_{1}^{\leqslant 0}=H_{1}$ is the heart of $D_{1}$, and in fact, $(U_{1}^{\leqslant 0}, U_{1}^{\geqslant 1})$ is t-structure in $H_{1}$ if and only if $(U_{1}^{\leqslant 0}, U_{1}^{\geqslant 1})$ is a torsion pair in $H_{1}$.
\end{example}

Before going to the main result of this section, we have a technical lemma by $D_{1}^{\geqslant 1} \supseteq D_{2}^{\geqslant 1}$.

\begin{lemma}
Let $D_{1}$ and $D_{2}$ be as above. For any $X \in \mathcal{D}$, we have
\begin{itemize}
\item[(1)]$(X_{2}^{\leqslant 0})_{1}^{\geqslant 1}\cong (X_{1}^{\geqslant 1})_{2}^{\leqslant 0}$
\item[(2)]$(X_{1}^{\leqslant 0})_{2}^{\geqslant 1}\cong (X_{2}^{\geqslant 1})_{1}^{\leqslant 0}\cong 0$
\item[(3)]$(X_{1}^{\leqslant 0})_{2}^{\leqslant 0}\cong X_{1}^{\leqslant 0} \cong (X_{2}^{\leqslant 0})_{1}^{\leqslant 0}$
\item[(4)]$(X_{1}^{\geqslant 1})_{2}^{\geqslant 1}\cong X_{2}^{\geqslant 1} \cong (X_{2}^{\geqslant 1})_{1}^{\geqslant 1}$
\end{itemize}
\end{lemma}
\begin{proof}
Let $X\in \mathcal{D}$. Since $\text{Hom}_{\mathcal{D}}(D_{1}^{\leqslant 0},D_{2}^{\geqslant 1})\subseteq \text{Hom}_{\mathcal{D}}(D_{2}^{\leqslant 0},D_{2}^{\geqslant 1})=0$, (2) is obvious.
\par Since $D_{1}^{\leqslant 0} \subseteq D_{2}^{\leqslant 0}$, $(X_{1}^{\leqslant 0})_{2}^{\leqslant 0}\cong X_{1}^{\leqslant 0}$ is clear. On the other hand, using the $D_{2}$-canonical distinguished triangle of $X$ and the octahedral axiom, we have the following commutative diagram:

$$\begin{tikzcd}
                                                                 & \Sigma^{-1}(X_{2}^{\geqslant 1}) \arrow[d] \arrow[r, "id"] & \Sigma^{-1}(X_{2}^{\geqslant 1}) \arrow[d]                &                                                                \\
(X_{2}^{\leqslant 0})_{1}^{\leqslant 0} \arrow[r] \arrow[d, "id"] & X_{2}^{\leqslant 0} \arrow[d] \arrow[r]                      & (X_{2}^{\leqslant 0})_{1}^{\geqslant 1} \arrow[d] \arrow[r] & \Sigma((X_{2}^{\leqslant 0})_{1}^{\leqslant 0}) \arrow[d, "id"] \\
(X_{2}^{\leqslant 0})_{1}^{\leqslant 0} \arrow[r]                 & X \arrow[r] \arrow[d]                                      & Y \arrow[d] \arrow[r]                                     & \Sigma((X_{2}^{\leqslant 0})_{1}^{\leqslant 0})                 \\
                                                                 & X_{2}^{\geqslant 1} \arrow[r, "id"]                        & X_{2}^{\geqslant 1}                                       &                                                               
\end{tikzcd}$$
which implies $Y \in D_{1}^{\geqslant 1}$, since $(X_{2}^{\leqslant 0})_{1}^{\geqslant 1} \in D_{1}^{\geqslant 1}$ and $X_{2}^{\leqslant 1} \in D_{2}^{\geqslant 1} \subseteq D_{1}^{\geqslant 1}$. Hence we have $(X_{2}^{\leqslant 0})_{1}^{\leqslant 0} \cong X_{1}^{\leqslant 0}$ and $Y \cong X_{1}^{\geqslant 1}$. Note $(X_{2}^{\leqslant 0})_{1}^{\geqslant 1} \in D_{2}^{\leqslant 0}$, since $X_{2}^{\leqslant 0}\in D_{2}^{\leqslant 0}$ and $\Sigma(X_{1}^{\leqslant 0}) \in D_{1}^{\leqslant -1} \subseteq D_{2}^{\leqslant 0}$. Hence we have  $(X_{2}^{\leqslant 0})_{1}^{\geqslant 1} \cong (X_{1}^{\geqslant 1})_{2}^{\leqslant 0}$. Therefore we finish the proof of (1) and (3), and the proof of (4) is similar.
 
\end{proof}

The main result of this section is the following.

\begin{theorem}\label{main theorem}
We denote $\mathcal{H}=D_{1}^{\geqslant 1} \bigcap D_{2}^{\leqslant 0} $. There is a correspondence:

$\begin{array}{*{3}{rll}}
\text{\{ t-structure \,} D_{i}=(D_{i}^{\leqslant 0},D_{i}^{\geqslant 1}) \text{\, in \,} D |\text{\,} D_{1}^{\leqslant 0} \subseteq D_{i}^{\leqslant 0} \subseteq D_{2}^{\leqslant 0}\} & \leftrightarrow & \{\text{t-structure in \,} \mathcal{H} \} \\
(D_{i}^{\leqslant 0},D_{i}^{\geqslant 1}) & \mapsto & (\mathcal{H} \bigcap D_{i}^{\leqslant 0},D_{i}^{\geqslant 1}\bigcap \mathcal{H})\\
(D_{1}^{\leqslant 0}\ast U_{i}^{\leqslant 0}, U_{i}^{\geqslant 1}\ast D_{2}^{\geqslant 1}) & \mapsfrom & (U_{i}^{\leqslant 0}, U_{i}^{\geqslant 1})

\end{array}$

\end{theorem}
\begin{proof}
Firstly, we check the maps are well-defined.
\par Let $D_{i}=(D_{i}^{\leqslant 0},D_{i}^{\geqslant 1})$ be a t-structure  of $\mathcal{D}$ with $D_{1}^{\leqslant 0} \subseteq D_{i}^{\leqslant 0} \subseteq D_{2}^{\leqslant 0}$, then we must check $(\mathcal{H} \bigcap D_{i}^{\leqslant 0},D_{i}^{\geqslant 1}\bigcap \mathcal{H})$ satisfies (T1'), (T2') and (T3').Observe that 
$$\text{Hom}_{\mathcal{D}}(\mathcal{H} \bigcap D_{i}^{\leqslant 0},D_{i}^{\geqslant 1}\bigcap \mathcal{H})\subseteq \text{Hom}_{\mathcal{D}}(D_{1}^{\leqslant 0},D_{1}^{\geqslant 1})=0,$$
thus (T2') is satisfied. To verify (T1'), we have the following commutative diagram with all rows and columns are distinguished triangles for any $X \in \mathcal{D}$
$$\begin{tikzcd}
\Sigma^{-1}(X_{1}^{\geqslant1})_{1}^{\leqslant 1} \arrow[d] \arrow[r] & 0 \arrow[d] \arrow[r] & (X_{1}^{\geqslant1})_{1}^{\leqslant 1} \arrow[d] \arrow[r] & (X_{1}^{\geqslant1})_{1}^{\leqslant 1} \arrow[d] \\
X_{1}^{\leqslant 0} \arrow[d] \arrow[r]                               & X \arrow[d] \arrow[r] & X_{1}^{\geqslant 1} \arrow[d] \arrow[r]                    & \Sigma(X_{1}^{\leqslant 0}) \arrow[d]            \\
X_{1}^{\leqslant 1} \arrow[r] \arrow[d]                               & X \arrow[r] \arrow[d] & X_{1}^{\geqslant 2} \arrow[r] \arrow[d]                    & \Sigma(X_{1}^{\leqslant 1}) \arrow[d]            \\
(X_{1}^{\leqslant 1})^{\geqslant 1} \arrow[r]                         & 0 \arrow[r]           & \Sigma(X_{1}^{\geqslant1})_{1}^{\leqslant 1} \arrow[r]     & \Sigma(X_{1}^{\geqslant1})_{1}^{\leqslant 1}    
\end{tikzcd}.$$
We find $\mathcal{H}\bigcap D_{i}^{\leqslant 0}=D_{1}^{\geqslant 1} \bigcap D_{i}^{\leqslant 0} $ and $\Sigma^{-1}(D_{i}^{\leqslant 0} \bigcap D_{1}^{\geqslant 1})=D_{i}^{\leqslant 2} \bigcap D_{1}^{\geqslant 1}$, so for any $X \in D_{1}^{\leqslant 0}\ast D_{i}^{\leqslant 0} \bigcap D_{1}^{\geqslant 1}$,we have $X_{1}^{\geqslant 1} \in D_{i}^{\leqslant 0} \bigcap D_{1}^{\geqslant 1} \subseteq D_{i}^{\leqslant 0}$, which implies that $X_{1}^{\geqslant 2} \in D_{i}^{\leqslant 0} \subset D_{i}^{\leqslant 1}$ and $X \in D_{1}^{\leqslant 1} \ast (D_{i}^{\leqslant 1} \bigcap D_{1}^{\geqslant 2})$. In the same way we can check $(D_{i}^{\geqslant 1} \bigcap D_{2}^{\leqslant 0})\ast D_{2}^{\geqslant 1} \subset (D_{i}^{\geqslant 0} \bigcap D_{2}^{\leqslant -1})\ast D_{2}^{\geqslant -1}$, therefore (T1') is satisfied.  Notice that we have the canonical distinguished triangle $X_{i}^{\leqslant 0} \rightarrow X \rightarrow X_{i}^{\geqslant 1} \rightarrow \Sigma(X_{i}^{\leqslant 0})$, hence for any $X \in \mathcal{H}$ we have $X_{i}^{\leqslant 0} \in D_{1}^{\geqslant 1} \bigcap D_{i}^{\leqslant 0}=\mathcal{H}\bigcap D_{i}^{\leqslant 0}$ and $X_{i}^{\geqslant 1} \in D_{i}^{\geqslant 1} \bigcap D_{2}^{\leqslant 0}=D_{i}^{\geqslant 1}\bigcap \mathcal{H}$.

\par Conversely, we need to check that $(D_{1}^{\leqslant 0}\ast U_{i}^{\leqslant 0}, U_{i}^{\geqslant 1}\ast D_{2}^{\geqslant 1})$ is a t-structure of $\mathcal{D}$ for any t-structure $(U_{i}^{\leqslant 0},U_{i}^{\geqslant 1})$ of $\mathcal{H}$.
\par By the definition of $(U_{i}^{\leqslant 0},U_{i}^{\geqslant 1})$, (T1) is satisfied. Notice 
$$\text{Hom}_{\mathcal{D}}(D_{1}^{\leqslant 0},U_{i}^{\geqslant 1})=\text{Hom}_{\mathcal{D}}(D_{1}^{\leqslant 0},D_{2}^{\geqslant 1})=0,$$
$$\text{Hom}_{\mathcal{D}}(U_{i}^{\leqslant 0},U_{i}^{\geqslant 1})=\text{Hom}_{\mathcal{D}}(U_{i}^{\leqslant 0},D_{2}^{\geqslant 1})=0,$$
thus $\text{Hom}_{\mathcal{D}}(D_{1}^{\leqslant 0} \ast U_{i}^{\leqslant 0},U_{i}^{\geqslant 1}\ast D_{2}^{\geqslant 1})=0$ and (T2') is satisfied. To show (T3'), we let $X \in \mathcal{D}$. We begin by using $D_{2}$-canonical distinguished triangle of $X$, $D_{1}$-canonical distinguished triangle of $X_{2}^{\leqslant 0}$ and corollary\ref{cor3x3}  to get the following commutative diagram and all rows and columns are distinguished triangles
$$\begin{tikzcd}
X_{1}^{\leqslant 0} \arrow[r, "id"] \arrow[d]     & X_{1}^{\leqslant 0} \arrow[d] \arrow[r] & 0 \arrow[d]                         \\
X_{2}^{\leqslant 0} \arrow[d] \arrow[r]           & X \arrow[d, "f"] \arrow[r]              & X_{2}^{\geqslant 1} \arrow[d, "id"] \\
(X_{2}^{\leqslant 0})_{1}^{\geqslant 1} \arrow[r] & P \arrow[r]                             & X_{2}^{\geqslant 1}                
\end{tikzcd}.$$
Note that we have a distinguish triangle 
$$X_{1}^{\leqslant 0} \rightarrow X \rightarrow P \rightarrow \Sigma(X_{1}^{\leqslant 0}).$$
Now we use the $U_{i}$-canonical distinguished triangle of $(X_{2}^{\leqslant 0})_{1}^{\geqslant 1}$ and corollary\ref{cor3x3} to form the following commutative diagram and all rows and columns are distinguished triangles
$$\begin{tikzcd}
((X_{2}^{\leqslant 0})_{1}^{\geqslant 1})_{U_{i}}^{\leqslant 0} \arrow[r, "id"] \arrow[d] & ((X_{2}^{\leqslant 0})_{1}^{\geqslant 1})_{U_{i}}^{\leqslant 0} \arrow[r] \arrow[d] & 0 \arrow[d]                   \\
(X_{2}^{\leqslant 0})_{1}^{\geqslant 1} \arrow[d] \arrow[r]                               & P \arrow[r] \arrow[d,"g"]                                                               & X_{2}^{\geqslant 1} \arrow[d, "id"]  \\
((X_{2}^{\leqslant 0})_{1}^{\geqslant 1})_{U_{i}}^{\geqslant 1} \arrow[r]                 & Q \arrow[r]                                                                         & X_{2}^{\geqslant 1}          
\end{tikzcd}$$

Then the lower distinguished triangle shows that $Q \in H_{2}^{\geqslant 1} \ast D_{2}^{\geqslant 1}$ and have a distinguished triangle
$$((X_{2}^{\leqslant 0})_{1}^{\geqslant 1})_{U_{i}}^{\leqslant 0} \rightarrow P \rightarrow Q \rightarrow\Sigma(((X_{2}^{\leqslant 0})_{1}^{\geqslant 1})_{U_{i}}^{\leqslant 0}).$$
Finally, we have the following commutative diagram and all rows and columns are distinguished triangles,
$$\begin{tikzcd}
X_{1}^{\leqslant 0} \arrow[r] \arrow[d, "id"] & R \arrow[r] \arrow[d]                   & ((X_{2}^{\leqslant 0})_{1}^{\geqslant 1})_{U_{i}}^{\leqslant 0}  \arrow[d] \\
X_{1}^{\leqslant 0} \arrow[r] \arrow[d]       & X \arrow[d, "g \circ f"] \arrow[r, "f"] & P \arrow[d, "g"]                                                           \\
0 \arrow[r]                                   & Q \arrow[r, "id"]                       & Q                                                                         
\end{tikzcd}$$
We observe that the upper distinguished triangle shows that $R \in D_{1}^{\leqslant 0} \ast U_{i}^{\leqslant 0}$ and we have the distinguished triangle
$$R \rightarrow X \rightarrow Q \rightarrow \Sigma(R)$$
which shows that (T3') is satisfied.

\par We show that the mappings are mutually inverse. 
\par Let $D_{i}$ be a t-structure of $\mathcal{D}$ such that $D_{1}^{\leqslant 0} \subseteq D_{i}^{\leqslant 0} \subseteq D_{2}^{\leqslant 0}$.Observe that $D_{1}^{\leqslant 0}$, $D_{i}^{\leqslant 0} \bigcap \mathcal{H}=D_{i}^{\leqslant 0} \bigcap D_{1}^{\geqslant 1} \subseteq D_{i}^{\leqslant 0}$, which means $D_{1}^{\leqslant 0} \ast (D_{i}^{\leqslant 0} \bigcap \mathcal{H}) \subseteq D_{i}^{\leqslant 0}$. On the other hand, for the any $X \in D_{i}^{\leqslant 0}$, we have the $D_{1}$-canonical distinguished triangle
$$X_{1}^{\leqslant 0} \rightarrow X \rightarrow X_{1}^{\geqslant 1} \rightarrow \Sigma(X_{1}^{\leqslant 0}),$$
which implies that $X_{1}^{\geqslant 1} \in D_{i}^{\leqslant 0} \bigcap D_{1}^{\geqslant 1}=D_{i}^{\leqslant 0} \bigcap \mathcal{H}$, hence $D_{1}^{\leqslant 0} \ast (D_{i}^{\leqslant 0} \bigcap \mathcal{H}) \supseteq D_{i}^{\leqslant 0}$. The fact that $D_{i}^{\geqslant 1}=(D_{i}^{\geqslant 1} \bigcap \mathcal{H})\ast D_{2}^{\geqslant 1}$ follows by a dual argument.

\par Let $U_{i}$ be a t-structure on $\mathcal{H}$. We claim that $U_{i}^{\leqslant 0}=(D_{1}^{\leqslant 0} \ast U_{i}^{\leqslant 0}) \bigcap \mathcal{H}$. Suppose $X \in (D_{1}^{\leqslant 0} \ast U_{i}^{\leqslant 0}) \bigcap \mathcal{H}$, then we have a distinguished triangle
$$Y \rightarrow X \rightarrow Z \rightarrow \Sigma(Y)$$
with $Y \in D_{1}^{\leqslant 0}$ and $Z \in U_{i}^{\leqslant 0}$. For any $W \in U_{i}^{\leqslant 0}$,  we apply the functor $\text{Hom}_{\mathcal{D}}(W,-)$, we have the exact sequence
$$\text{Hom}_{\mathcal{D}}(W,Y) \rightarrow \text{Hom}_{\mathcal{D}}(W,X) \rightarrow \text{Hom}_{\mathcal{D}}(W,Z) \rightarrow \text{Hom}_{\mathcal{D}}(W,\Sigma(Y)).$$
Notice $\text{Hom}_{\mathcal{D}}(W,Y) \in \text{Hom}_{\mathcal{D}}(D_{1}^{\leqslant 0}, \mathcal{H})=0$ and $\text{Hom}_{\mathcal{D}}(W,Z) \in \text{Hom}_{\mathcal{D}}(U_{i}^{\leqslant 0}, U_{i}^{\geqslant 1})=0$ for any $W \in U_{i}^{\leqslant 0}$, therefore we have $X \in U_{i}^{\leqslant 0}$. Conversely, let $M \in U_{i}^{\leqslant 0}$, we have the the $D_{1}$-canonical distinguished triangle of $M$ 

$$M_{1}^{\leqslant 0} \rightarrow M \rightarrow M_{1}^{\geqslant 1} \rightarrow \Sigma(M_{1}^{\leqslant 0}).$$
By applying $\text{Hom}_{\mathcal{D}}(-,N)$ to the $D_{1}$-canonical distinguished triangle of $M$ for any $N \in U_{i}^{\geqslant 1}$, we have a exact sequence
$$\text{Hom}_{\mathcal{D}}(\Sigma(M_{1}^{\leqslant 0}),N) \rightarrow \text{Hom}_{\mathcal{D}}(M_{1}^{\geqslant 1},N) \rightarrow \text{Hom}_{\mathcal{D}}(M,N)  \rightarrow \text{Hom}_{\mathcal{D}}(M_{1}^{\leqslant 0},N),$$
which means that $M_{1}^{\geqslant 1} \in \mathcal{H}$. Observe $\text{Hom}_{\mathcal{D}}(\Sigma(M_{1}^{\leqslant 0}),N) \in \text{Hom}_{\mathcal{D}}(D_{1}^{\leqslant 0},D_{1}^{\geqslant 1})=0$ and $\text{Hom}_{\mathcal{D}}(M,N) \in \text{Hom}_{\mathcal{D}}(U_{i}^{\leqslant 0},U_{i}^{\geqslant 1})=0$ for any $N \in U_{i}^{\geqslant 1}$, so we have $M \in (D_{1}^{\leqslant 0} \ast U_{i}^{\leqslant 0}) \bigcap \mathcal{H}$. The fact that $U_{i}^{\geqslant 1}=(U_{i}^{\geqslant 1} \ast D_{2}^{\geqslant 1}) \bigcap \mathcal{H}$ holds follows by a dual argument.

\end{proof}

\begin{corollary}[\cite{happel1996tilting} Proposition 2.1, \cite{beligiannis2007homological} Theorem 3.1]
Let $D_{1}$ be a t-structure on a triangulated category $\mathcal{D}$. Then there is a canonical isomorphism between the torsion pair in the heart $H_{1}$ and the t-structure $D_{2}$ satisfied $D_{1}^{\leqslant 0} \subseteq D_{2}^{\leqslant 0} \subseteq D_{1}^{\leqslant 1}$. 
\end{corollary}
\begin{proof}
Let the theorem\ref{main theorem} for $(D_{1}^{\leqslant 0},D_{1}^{\geqslant 1})$ and $(D_{1}^{\leqslant 1},D_{1}^{\geqslant 2} )$, then the t-structure $D_{2}$ satisfied $D_{1}^{\leqslant 0} \subseteq D_{2}^{\leqslant 0} \subseteq D_{1}^{\leqslant 1}$ and torsion pair in $\Sigma(H_{1})$ correspond one to one, where $\Sigma(H_{1})\cong H_{1}$, so we finish the proof.
\end{proof}

\section{The distance between t-structures}\label{section distance}
In this section, we discuss the distance between two t-structure and determine their distance, and then 
\begin{definition}{(\cite{chen2022extensions}, Definition 4.2)}\label{t-structure distance}
For two t-structures $D_{1}$ and $D_{2}$, we define their \textbf{distance} $d(D_{1},D_{2})$ to be the smallest natural number $d$ such that $D_{1}^{\leqslant m} \subseteq D_{2}^{\leqslant 0} \subseteq D_{1}^{\leqslant m+d}$ for some integer $m$. If such $d$ does not exist, we set $d(D_{1},D_{2})=+ \infty.$
\end{definition}

Applying the translation functor, we have the following basic fact:

\begin{lemma}\label{sy}
Let $D_{1}=(D_{1}^{\leqslant 0},D_{1}^{\geqslant 1})$ and $D_{2}=(D_{2}^{\leqslant 0},D_{2}^{\geqslant 1})$ be two t-structures on $\mathcal{D}$. Assume that $m$ and $n$ are integers satisfying $m \leqslant n$. Then the following statements are equivalent:
\begin{itemize}
\item[(1)] $D_{1}^{\leqslant m} \subseteq D_{2}^{\leqslant 0} \subseteq D_{1}^{\leqslant n}$;
\item[(2)] $D_{2}^{\leqslant -n} \subseteq D_{1}^{\leqslant 0} \subseteq D_{2}^{\leqslant -m}$;
\item[(3)] $D_{1}^{\geqslant n} \supseteq D_{2}^{\geqslant 0} \supseteq D_{1}^{\geqslant m}$;
\item[(4)] $D_{2}^{\geqslant -m} \supseteq D_{1}^{\geqslant 0} \supseteq D_{2}^{\geqslant -n}$
\item[(5)] $D_{2}^{\leqslant m} \subseteq D_{1}^{\leqslant 0}$ and $D_{2}^{\geqslant n} \subseteq D_{1}^{\geqslant 0} $;
\item[(6)] $D_{1}^{\leqslant -n} \subseteq D_{2}^{\leqslant 0}$ and $D_{1}^{\geqslant -m} \subseteq D_{2}^{\geqslant 0} $.
\end{itemize}
\end{lemma}

\begin{proposition}
We have $d(D_{1},D_{2})=d(D_{2},D_{1})$, and $d(D_{1},D_{2})=0$ if and only if $D_{2}=\Sigma^{n}(D_{1})$ for some integer $n$.
\end{proposition}
\begin{proof}
By definition, it's easy to check the second statement, hence we only show $d(D_{1},D_{2})=d(D_{2},D_{1})$. 
\par Let $d=d(D_{1},D_{2})$ and $d'=d(D_{2},D_{1})$. By definition, we have $D_{1}^{\leqslant m} \subseteq D_{2}^{\leqslant 0} \subseteq D_{1}^{\leqslant m+d}$ for some integer $m$, which implies $D_{2}^{\leqslant -(m+d)} \subseteq D_{1}^{\leqslant 0} \subseteq D_{1}^{\leqslant -m}$ and $d \geqslant d'$. Conversely, we have $d \leqslant d'$ and hence $d=d'$.
\end{proof}

Let $D_{1}$ and $D_{2}$ be two t-structures on $\mathcal{D}$ with $d(D_{1},D_{2})=d < \infty$. A natural question is whether the $m$ in the definition of distance is unique when $d$ is fixed. The following proposition give a positive answer.

\begin{proposition}\label{unique}
Let $D_{1}=(D_{1}^{\leqslant 0}, D_{1}^{\geqslant 1})$ and $D_{2}=(D_{2}^{\leqslant 0}, D_{2}^{\geqslant 1})$ be two t-structures in $\mathcal{D}$. Let $d$ be the smallest natural number such that $D_{1}^{\leqslant m} \subseteq D_{2}^{\leqslant 0} \subseteq D_{1}^{\leqslant m+d}$ for some integer $m$, then the $m$ is unique.
\end{proposition}
\begin{proof}
Let $n$ be another natural number satisfies the condition, then we must show $m=n$. Without loss of generality we suppose $m\leqslant n$.
\par If $m+d\leqslant n$, then we have $D_{1}^{\leqslant m} \subseteq D_{2}^{\leqslant 0} \subseteq D_{1}^{\leqslant m+d} \subseteq D_{1}^{\leqslant n} \subseteq D_{2}^{\leqslant 0} \subseteq D_{1}^{\leqslant n+d}$, which means $D_{2}^{\leqslant 0}=D_{1}^{\leqslant n}$ and $d=0$. Hence we have $m=n$.
\par If $m+d > n$, then we have $D_{1}^{\leqslant m} \subseteq D_{2}^{\leqslant 0} \subseteq D_{1}^{\leqslant m+d}$ and $D_{1}^{\leqslant n} \subseteq D_{2}^{\leqslant 0} \subseteq D_{1}^{\leqslant n+d}$, which implies that $D_{1}^{\leqslant n} \subseteq D_{2}^{\leqslant 0} \subseteq D_{1}^{\leqslant m+d}$ and we must have $m=n$ by the choice of $d$.
\end{proof}

Therefore, if $d(D_{1},D_{2})=d < \infty$, then the $d$ and $m$ are unique, and we denote the $m$ with $m_{d}$. Hence, one main problem is that how to determine $d$ or even the $m_{d}$? We will give a positive answer to this question. Before going deeper, let's fix some notations to help us to state the conclusion and some lemmas that would be used later. 
\par In the rest of this section, we always let $D_{1}=(D_{1}^{\leqslant 0}, D_{1}^{\geqslant 1})$ and $D_{2}=(D_{2}^{\leqslant 0}, D_{2}^{\geqslant 1})$ be two t-structures on $\mathcal{D}$ with $d(D_{1},D_{2})< \infty$. We assume $D_{1}^{\leqslant n_{1}} \subseteq D_{2}^{\leqslant 0} \subseteq D_{1}^{\leqslant n_{2}}$ for some integers number $n_{1}$ and $n_{2}$ with $n_{1}\leqslant n_{2}$. If either $D_{1}$ or $D_{2}$ is stable or in the case of $n_{1}=n_{2}$, then we must have $d(D_{1},D_{2})=0$, and this case is trivial, therefore we also suppose neither $D_{1}$ nor $D_{2}$ is stable and $n_{1}+1 \leqslant n_{2}$. Then we can find the relationship among $m_{d}$, $d$, $n_{1}$ and $n_{2}$.

\begin{lemma}
For $D_{1}=(D_{1}^{\leqslant 0}, D_{1}^{\geqslant 1})$ and $D_{2}=(D_{2}^{\leqslant 0}, D_{2}^{\geqslant 1})$, if  $D_{1}^{\leqslant n_{1}} \subseteq D_{2}^{\leqslant 0} \subseteq D_{1}^{\leqslant n_{2}}$, then we have $n_{1} \leqslant m_{d} \leqslant m_{d}+d \leqslant n_{2}$.
\end{lemma}
\begin{proof}
First, we claim that $m_{d}+d \geqslant n_{1}$. If not, then we have 
$$D_{1}^{\leqslant m_{d}} \subseteq D_{2}^{\leqslant 0} \subseteq D_{1}^{\leqslant m_{d}+d} \subseteq D_{1}^{\leqslant n_{1}} \subseteq D_{2}^{\leqslant 0} \subseteq D_{1}^{\leqslant n_{2}}.$$
It proves that $D_{2}^{\leqslant 0}=D_{1}^{\leqslant m_{d}+d}=D_{1}^{\leqslant n_{1}}$ and $n_{1}=m_{d}+d$, which is a contradiction. By the same way we can check that $m_{d} \leqslant n_{2}$.
\par And we also claim that $n_{1} \leqslant m_{d}$. If not, we have $D_{1}^{\leqslant n_{1}} \subseteq D_{2}^{\leqslant 0} \subseteq D_{1}^{m_{d}+d}$, and then $m_{d}+d-n_{1} <d$, which is a contradiction to definition of $d$. By the same way we can check that $m_{d}+d \leqslant n_{2}$.
\end{proof}

Hence, in our suppose, we just need to analyse when $D_{2}^{\leqslant 0}$ can be embedded in a 'subintermediate', i.e. $D_{1}^{k_{1}}\subseteq D_{2}^{\leqslant 0} D_{1}^{\leqslant k_{2}}$ for some $n_{1}\leqslant k_{1} \leqslant k_{2} \leqslant n_{2}$. In fact, we have the following notion which is known to experts; compare (\cite{marks2021lifting}, Subsection 2.2).
\begin{definition}
Let $D_{1}=(D_{1}^{\leqslant 0}, D_{1}^{\geqslant 1})$ be a t-structure on $\mathcal{D}$ and $n_{1}$, $n_{2}$ be integers with $n_{1} \leqslant n_{2}$. We define $D_{1}^{[n_{1},n_{2}]}= D_{1}^{\geqslant n_{1}} \bigcap D_{1}^{\leqslant n_{2}}$, called the \textbf{intermediate} with respect to $D_{1}$ from $n_{1}$ to $n_{2}$. For any integers $a, b$ satisfying $n_{1} \leqslant a \leqslant b \leqslant n_{2}$, $D_{1}^{[a,b]}$ is called a \textbf{subintermediate} of $D_{1}^{[n_{1},n_{2}]}$. 
\end{definition}

In fact, for every intermediate with respect to $D_{1}$ from $n_{1}$ to $n_{2}$, we have the following lemma.

\begin{lemma}\label{key lemma}
Let $D_{1}=(D_{1}^{\leqslant 0}, D_{1}^{\geqslant 1})$ be a t-structure on $\mathcal{D}$ and $n_{1}, n_{2}$ be integers with $n_{1} \leqslant n_{2}$. Then we have $D_{1}^{\leqslant n_{2}}=D_{1}^{\leqslant n_{1}} \ast \Sigma^{-(n_{1}+1)}(H_{1})\ast \cdots \ast \Sigma^{-n_{2}}(H_{1})$, and $D_{1}^{\geqslant n_{1}}=\Sigma^{-n_{1}}(H_{1})\ast \cdots \ast \Sigma^{-(n_{2}-1)}(H_{1})\ast D_{1}^{\geqslant n_{2}}$. In particular, $D_{1}^{[n_{1},n_{2}]}= \Sigma^{-n_{1}}(H_{1})\ast \cdots \ast  \Sigma^{-n_{2}}(H_{1})$.
\end{lemma}
\begin{proof}
If $n_{1}=n_{2}$, the statement is trivial, hence we only consider the case $n_{1} < n_{2}$.
\par To begin with, let $X \in D_{1}^{\leqslant n_{2}}$. We have the $(D_{1}^{\leqslant n_{2}-1},D_{1}^{\geqslant n_{2}})$-canonical distinguished triangle of $X$
$$X_{1}^{\leqslant n_{2}-1} \rightarrow X \rightarrow X_{1}^{\geqslant n_{2}} \rightarrow \Sigma(X_{1}^{\leqslant n_{2}-1}).$$
Since $X \in D_{1}^{\leqslant n_{2}}$, we have $X \cong X_{1}^{\leqslant n_{2}} $, which implies $(X_{1}^{\leqslant n_{2}})_{1}^{\geqslant n_{2}} \cong (X_{1}^{\geqslant n_{2}})_{1}^{\leqslant n_{2}} \in \Sigma^{-n_{2}}(H_{1})$. Then we have the $(D_{1}^{\leqslant n_{1}-2,},D_{1}^{\geqslant n_{1}-1})$-canonical distinguished triangle of $X_{1}^{\leqslant n_{2}-1}$
$$(X_{1}^{\leqslant n_{2}-1})_{1}^{\leqslant n_{2}-2} \rightarrow X_{1}^{\leqslant n_{2}-1} \rightarrow (X_{1}^{\leqslant n_{2}-1})_{1}^{\geqslant n_{2}-1} \rightarrow \Sigma((X_{1}^{\leqslant n_{2}-1})_{1}^{\leqslant n_{2}-2}).$$
Note that $(X_{1}^{\leqslant n_{2}-1})_{1}^{\geqslant n_{2}-1}\cong (X_{1}^{\geqslant n_{2}-1})_{1}^{\leqslant n_{2}-1} \in \Sigma^{-(n_{2}-1)}(H_{1})$ and $(X_{1}^{\leqslant n_{2}-1})_{1}^{\leqslant n_{2}-2}\cong X_{1}^{n_{2}-2}$. By the same process, we have the distinguished triangle 
$$(X_{1}^{\leqslant n_{1}+1})_{1}^{\leqslant n_{1}}\rightarrow X_{1}^{\leqslant n_{1}+1} \rightarrow (X_{1}^{\leqslant n_{1}+1})_{1}^{\geqslant n_{1}+1} \rightarrow \Sigma((X_{1}^{\leqslant n_{1}+1})_{1}^{\leqslant n_{1}}),$$
where $(X_{1}^{\leqslant n_{1}+1})_{1}^{\leqslant n_{1}} \in D_{1}^{\leqslant n_{1}}$ and $(X_{1}^{\leqslant n_{1}+1})_{1}^{\geqslant n_{1}+1} \in \Sigma^{-(n_{1}+1)}$. Hence $D_{1}^{\leqslant n_{2}}\subseteq D_{1}^{\leqslant n_{1}} \ast \Sigma^{-(n_{1}+1)}(H_{1})\ast \cdots \ast \Sigma^{-n_{2}}(H_{1})$. Conversely, $D_{1}^{\leqslant n_{1}} \subseteq D_{1}^{\leqslant n_{2}}$ and $\Sigma^{-i}(H_{1}) \subseteq D_{1}^{\leqslant n_{2}}$ for any integer $i \in \{n_{1}+1,\cdots,n_{2} \}$, hence $D_{1}^{\leqslant n_{2}}=D_{1}^{\leqslant n_{1}} \ast \Sigma^{-(n_{1}+1)}(H_{1})\ast \cdots\ast \Sigma^{-n_{2}}(H_{1})$. And $D_{1}^{\geqslant n_{1}}=\Sigma^{-n_{1}}(H_{1})\ast\cdots\ast \Sigma^{-(n_{2}-1)}(H_{1})\ast D_{1}^{\geqslant n_{2}}$ is just the dual statement.
\par According to remark\ref{useful remark} and associative law of $\ast$, the decomposition above is unique, hence $D_{1}^{[n_{1},n_{2}]}= \Sigma^{-n_{1}}(H_{1})\ast\cdots\ast  \Sigma^{-n_{2}}(H_{1})$.
%$i \in \{n_{1},n_{1}+1,...,n_{2}\}$, $\Sigma^{-i} \subseteq D_{1}^{[n_{1},n_{2}]}$, hence $D_{1}^{[n_{1},n_{2}]} \supseteq \Sigma^{-n_{1}}(H_{1})\ast \Sigma^{-(n_{1}+1)}(H_{1})\ast ...\ast \Sigma^{-n_{2}}(H_{1}).$

\end{proof}

By the theorem\ref{main theorem}, we know 
$$U=(U^{\leqslant 0},U^{\geqslant 1})=(D_{1}^{[n_{1},n_{2}]}\cap D_{2}^{\leqslant 0}, D_{2}^{\geqslant 1}\cap D_{1}^{[n_{1},n_{2}]})$$
is a t-structure on $D_{1}^{[n_{1},n_{2}]}$. Following lemma\ref{key lemma}, we know $D_{1}^{[n_{1},n_{2}]}=\Sigma^{-n_{1}}(H_{1})\ast\cdots\ast\Sigma^{-n_{2}}(H_{1})$ .Now, let's turn back to the main question. If there exist $i \in \{n_{1}+1,\cdots,n_{2}\}$ such that for any $n_{1}+1 \leqslant j\leqslant i $(resp.$i \leqslant j \leqslant n_{2}$), we have $\Sigma^{-j}(H) \subseteq U^{\leqslant 0}$(resp. $\Sigma^{-j}(H) \subseteq U^{\geqslant 1}$) , then we suppose that $a$ (resp. $b$) is the largest(resp. smallest) one; otherwise, we set $a=n_{1}$ and $b=n_{2}+1$. Notice that we always have $a \leqslant b-1$.

\begin{theorem}\label{second theorem}
Following all the notion above, we have $d=b-1-a$ and $m_{d}=a$.
\end{theorem}
\begin{proof}
Firstly, we consider the case that those $i$ exist. Then we know $\Sigma^{-n_{1}}(H_{1})\ast...\Sigma^{-a}(H_{1}) \subseteq U^{\leqslant 0}$ and $\Sigma^{-b}(H_{1})\ast..\ast \Sigma^{-n_{2}}(H_{1}) \subseteq U^{\geqslant 1}$, since $U^{\leqslant 0}$ and $U^{\geqslant 1}$ are closed under extension. Then 
$$D_{1}^{\leqslant a}=D_{1}^{\leqslant n_{1}}\ast \Sigma^{-(n_{1}+1)}(H_{1})\ast\cdots\ast \Sigma^{-a}(H_{1}) \subseteq D_{2}^{\leqslant 0}$$
and 
$$D_{1}^{\geqslant b}=\Sigma^{-b}(H_{1})\ast\cdots\ast \Sigma^{-n_{2}}(H_{1})\ast D_{1}^{\geqslant n_{2}+1} \subseteq D_{2}^{\geqslant 1}.$$
Hence $D_{1}^{\leqslant a} \subseteq D_{2}^{\leqslant 0} \subseteq D_{1}^{\leqslant b-1}$, which implies that $d \leqslant b-1-a$ and $m_{d} \geqslant a$. If $d < b-1-a$ and $m_{d} > a$, then $\Sigma^{-m_{d}}(H_{1})\subseteq U^{\leqslant 0}$ and $\Sigma^{-(m_{d}+d)}(H_{1})\subseteq U^{\geqslant 1}$, which is a contradiction.
\par It's the same process for the case $a=n_{1}$ and $b=n_{2}+1$, and we finish the proof.
\end{proof}

Theorem \ref{second theorem} tells us when the t-structure can contract to a subintermediate. In fact, using theorem\ref{main theorem} and theorem\ref{second theorem} we can have the following proposition.

\begin{proposition}
Let $D_{1}=(D_{1}^{\leqslant 0}, D_{1}^{\geqslant 1})$ be a t-structure satisfying $H_{1}$ is non-zero and there is a non-trivial torsion pair $(T,F)$ in $H_{1}$. Then for any positive natural number $n$, we have a t-structure $D_{2}=(D_{2}^{\leqslant 0}, D_{2}^{\geqslant 1})$ such that $d(D_{1},D_{2})=n$.
\end{proposition}
\begin{proof}
The case that $n=0$ is trivial, hence we consider when $n\geqslant 1$.In fact, for any positive natural $n$, the pair of full subcategories 
$$U=(U^{\leqslant 0},U^{\geqslant 1})=(\Sigma^{-1}(T)\ast\cdots\ast \Sigma^{-n}(T), F\ast\cdots\ast \Sigma^{-n}(F))$$
is a t-structure on $D_{1}^{[1,n]}$. We finish it by induction.
\begin{itemize}
\item[(1)]for n=1: $D_{1}^{\leqslant -1} \ast T\subseteq  D_{1}^{\leqslant 0} \ast \Sigma^{-1}(T)$, since $D_{1}^{\leqslant -1},T \subseteq D_{1}^{\leqslant 0}$ and $D_{1}^{\leqslant 0}$ is closed under extension; (T2') and (T3') are obvious, since $D_{1}^{[1]}=\Sigma^{-1}(H_{1})$, and $(\Sigma^{-1}(T),\Sigma^{-1}(F))$ is torsion pair on $\Sigma^{-1}(H_{1})$. 

\item[(2)] assume that $(\Sigma^{-1}(T)\ast \cdots \ast \Sigma^{-n}(T), \Sigma^{-1}(F)\ast\cdots\ast \Sigma^{-n}(F)$ is a t-structure on $D_{1}^{[1,n]}$, then we have
$$D_{1}^{\leqslant -1}\ast T\ast\cdots\ast \Sigma^{-n+1}(T)\subseteq D_{1}^{\leqslant 0}\ast \Sigma^{-1}(T)\ast \cdots \ast \Sigma^{-n}(T),$$
and
$$F\ast\cdots \ast \Sigma^{-(n-1)}(T) \ast D_{1}^{\geqslant n} \supseteq \Sigma^{-1}(F)\ast \cdots \ast \Sigma^{-n}(F)\ast D_{1}^{\geqslant n+1},$$
which means that 
$$D_{1}^{\leqslant -1}\ast T\ast\cdots \ast \Sigma^{-n+1}(T)\ast \Sigma^{-n+2}(H_{1}) \subseteq D_{1}^{\leqslant 0}\ast \Sigma^{-1}(T)\ast\cdots \ast \Sigma^{-n}(T)\ast \Sigma^{-n+1}(T),$$
and
$$F\ast\cdots \ast \Sigma^{-(n-1)}(F)\ast \Sigma^{-n}(F) \ast D_{1}^{\geqslant n+1} \supseteq \Sigma^{-1}(F)\ast\cdots \ast \Sigma^{-n}(F)\ast \Sigma^{-n+1}(F)\ast D_{1}^{\geqslant n+2},$$
hence we finish (T1'). In addition, 
$$\text{Hom}_{\mathcal{D}}(D_{1}^{\leqslant 0}\ast \Sigma^{-1}(T)\ast \cdots \ast \Sigma^{-n}(T)\ast \Sigma^{-n+1}(T),\Sigma^{-1}(F)\ast\cdots \ast \Sigma^{-n}(F)\ast \Sigma^{-n+1}(F)\ast D_{1}^{\geqslant n+2})=0,$$
is easy to check. 
\par For any $X \in D_{1}^{[1,n+1]}$, there exist $Y \in D_{1}^{[1,n]}$, $Z \in H_{1}$ such that $Y \rightarrow X \rightarrow \Sigma^{-n+1}(Z) \rightarrow \Sigma(Y)$ is a distinguished triangle. Then we have the following commutative diagram
$$\begin{tikzcd}
Y_{1} \arrow[r] \arrow[d] & U_{0} \arrow[d, dotted] \arrow[r] & \Sigma^{-n+1}(Z_{T}) \arrow[d] \\
Y \arrow[r] \arrow[d]     & X \arrow[r] \arrow[d, dotted]     & \Sigma^{-n+1}(Z) \arrow[d]     \\
Y_{2} \arrow[r]           & U_{1} \arrow[r]                   & \Sigma^{-n+1}(Z_{F})          
\end{tikzcd}$$
where all rows and columns are distinguished triangle.
\end{itemize}
Hence $U=(U^{\leqslant 0},U^{\geqslant 1})$ is indeed a t-structure on $D_{1}^{[1,n]}$. Following theorem\ref{main theorem}, we know $D_{2}=(D_{2}^{\leqslant 0},D_{2}^{\geqslant 1})=(D_{1}^{\leqslant 0}\ast U^{\leqslant 0},U^{\geqslant 1}\ast D_{1}^{\leqslant n+1})$ is a t-structure satisfying $D_{1}^{\leqslant 0} \subseteq D_{2}^{\leqslant 0} \subseteq D_{1}^{\leqslant n}$, and by theorem\ref{second theorem} we know $d(D_{1},D_{2})=n$
\end{proof}

We finish this section in a example. 
\begin{example}
Let $k$ be a field and $Q$ be the quiver 
$$\begin{tikzcd}
1 \arrow[r] & 2 \arrow[r] & 3
\end{tikzcd}.$$
Let $A=kQ$.Then $A$ has finite representation type, and all the indecomposable finitely generated left modules are as follow 
$$\begin{tikzcd}
1 \arrow[d] & 0 \arrow[d] & 0 \arrow[d] & 1 \arrow[d] & 1 \arrow[d] & 0 \arrow[d] \\
1 \arrow[d] & 1 \arrow[d] & 0 \arrow[d] & 1 \arrow[d] & 0 \arrow[d] & 1 \arrow[d] \\
1           & 1           & 1           & 0           & 0           & 0           \\
P_{3}       & P_{2}       & P_{1}       & I_{2}       & S_{1}       & S_{2}      
\end{tikzcd}$$
Let $\mathcal{T}=add(P_{3})$ and $\mathcal{F}=add(P_{2}\bigoplus I_{2} \bigoplus S_{1} \bigoplus S_{2})$ , then $(\mathcal{T},\mathcal{F})$ is a non-trivial torsion pair on $A$-mod. Let $D(A-mod)$ be the derived category of $A$-mod, then $D(A-mod)$ has a canonical t-structure $D_{1}=(D_{1}^{\leqslant 0}, D_{1}^{\geqslant 1})$ where 
$$D_{1}^{\leqslant 0}=\{X \in D(A-mod)\,|\, H^{i}(X)=0 \text{\; for any\;}  i \geqslant 1 \}$$
and 
$$D_{1}^{\geqslant 1}=\{X \in D(A-mod)\,|\, H^{i}(X)=0 \text{\; for any\;}  i \leqslant 0 \}.$$
For any natural number $n$, 
$$D_{2}=(D_{2}^{\leqslant 0},D_{2}^{\geqslant 1})=(D_{1}^{\leqslant -1} \ast \mathcal{T}\ast\cdots\ast \Sigma^{-n}(\mathcal{T}), \mathcal{F}\ast\cdots\ast \Sigma^{-n)}(\mathcal{F}) \ast D_{1}^{\geqslant n+1})$$
is a t-structure on $D(A-mod)$ with $d(D_{1},D_{2})=n$. In addition, we can also construct t-structure on $D(A-Mod)$ using the bijection in \cite{marks2021lifting} Section 4.
\end{example}

\section*{Acknowledgement}
I gratefully appreciate the help provided by Xiao Hu and Xiaohu Chen. For Xiao Hu, he is one of my best friends who teaches me so much math knowledge, and the main motivation of this paper comes  from my discussion with him.  For Xiaohu Chen, he is also one of my best friends, and I express sincere gratitude to him for spending time in checking some details in this paper.

\bibliographystyle{plain}
\bibliography{correspondence.bib}

\Addresses

\end{document}